\documentclass[12pt,a4paper]{article}
\usepackage{latexsym,amssymb,amsfonts,amsmath,amsthm,nccmath,enumitem,setspace,graphicx,bm,float,diagbox}
\usepackage[dvipsnames]{xcolor}
\usepackage[hidelinks]{hyperref}
\usepackage[margin=2cm]{geometry}

\newtheorem{theorem}{Theorem}[section]

\newtheorem{lemma}[theorem]{Lemma}

\newtheorem{claim}[theorem]{Claim}

\theoremstyle{definition}
\newtheorem{definition}[theorem]{Definition}

\newcommand{\R}{\mathbb{R}}

\def \deg {{\rm deg}}

\def \leq {\leqslant}
\def \geq {\geqslant}

\def \mod{\pmod}

\let\oldproofname=\proofname
\renewcommand{\proofname}{\rm\bf{\oldproofname}}

\title{Improved upper bounds on Zarankiewicz numbers}
\author{Sara Davies\thanks{School of Mathematics and Physics, The University of Queensland, St Lucia QLD 4067, Australia.}\qquad Peter Gill\thanks{School of Mathematics, Monash University, Melbourne VIC 3800, Australia.} \qquad Daniel Horsley$^\dag$}
\date{}

\begin{document}
\maketitle
\setstretch{1.2}

\begin{abstract}
For positive integers $s,t,m$ and $n$, the Zarankiewicz number $z(m,n;s,t)$ is the maximum number of edges in a subgraph of $K_{m,n}$ that has no complete bipartite subgraph containing $s$ vertices in the part of size $m$ and $t$ vertices in the part of size $n$. The best general upper bound on Zarankiewicz numbers is a bound due to Roman that can be viewed as the optimal value of a simple linear program. Here we show that in many cases this bound can be improved by adding additional constraints to this linear program. This allows us to prove new upper bounds on Zarankiewicz numbers for many small parameter sets. We are also able to establish a new family of closed form upper bounds on $z(m,n;s,t)$ that captures much, but not all, of the power of the new constraints. This bound generalises a recent result of Chen, Horsley and Mammoliti that applied only in the case $s=2$.

\bigskip

\end{abstract}

\section{Introduction} \label{S:Intro}

For positive integers $s$, $t$, $m$ and $n$, the \emph{Zarankiewicz number} $z(m,n;s,t)$ is usually defined to be the maximum number of edges in a bipartite graph with parts of sizes $m$ and $n$ that has no complete bipartite subgraph containing $s$ vertices in the part of size $m$ and $t$ vertices in the part of size $n$. The problem of finding these numbers was first posed by Zarankiewicz in \cite{Zar}.

A \emph{hypergraph} $H$ consists of a finite set of vertices $V(H)$ and a multiset $E(H)$ of edges where each edge is a nonempty subset of $V$. By considering the bipartite graph in the definition of a Zarankiewicz number as the vertex-edge incidence graph of a hypergraph, $z(m,n;s,t)$ can be equivalently defined as the maximum total degree of a hypergraph with $m$ vertices, $n$ edges, and with the property that each set of $s$ vertices is a subset of at most $t-1$ edges. We call a hypergraph with this property \emph{$(s,t-1)$-linear}, noting that $(2,1)$-linearity corresponds to the usual definition of linearity. Since replacing an edge of size less than $s-1$ with an arbitrary edge of size $s-1$ cannot affect whether a hypergraph is $(s,t-1)$-linear, we can restrict ourselves to hypergraphs in which each edge has size at least $s-1$ in the hypergraph definition of $z(m,n;s,t)$, and it will be convenient to do so throughout.

Roman \cite{Rom} established the following upper bounds on $z(m,n;s,t)$.

\begin{theorem}[\cite{Rom}]\label{T:Rom}
For positive integers $k$, $s$, $t$, $m$ and $n$ with $s \geq 2$ and $k \geq s-1$
\begin{equation}  \label{E:RomanBound}
z(m,n;s,t) \leq \frac{(t - 1) \binom{m}{s}}{ \binom{k}{s-1} }  + \mfrac{ (k+1)(s-1) }{s}n.
\end{equation}
\end{theorem}

Roman's proof implicitly shows that these bounds arise as the optimal values of the linear program with variables $n_{s-1},\ldots,n_m$ that maximises $\sum_{i=s-1}^mn_i$ subject to the constraints
\[
\sum_{i=s-1}^m n_i = n  \qquad \text{and} \qquad
\sum_{i=s-1}^m \mbinom{i}{s} n_i \leq (t-1) \mbinom{m}{s}.
\]
If $H$ is an $(s,t-1)$-linear hypergraph on $m$ vertices with no edges of size less than $s-1$ and $n_i$ edges of size $i$ for each $i \in \{s-1,\ldots,m\}$, then it is not difficult to see that the two constraints above hold (for details see Lemma~\ref{L:obeyLP}). Thus the optimal value of this linear program must be an upper bound for $z(m,n;s,t)$.

In \cite{CheHorMam} it was shown that, in the case $s=2$, one of a family of extra constraints could be added to this linear program and that the resulting upper bound improved Roman's bound in some cases. Our aim in this paper is to generalise this upper bound to larger values of $s$. To do this we generalise the additional constraints introduced in \cite{CheHorMam} as follows.

\begin{theorem}\label{T:constraint}
Let $s,t \geq 2$ be integers. Let $H$ be an $(s,t-1)$-linear hypergraph with $m$ vertices and $n$ edges such that $H$ has exactly $n_i$ edges of size $i$ for each $i \in \{s-1,\ldots,m\}$, and no edges of size less than $s-1$. For any positive integers $v$ and $k$ such that $v < s \leq k \leq m$, we have
\[
\frac{1}{ \binom{k-v}{s-v} - \alpha } \sum_{i=s-1}^{k-1} \left( \mbinom{i-v}{s-v} - \alpha \right) \mbinom{i}{v} n_i + \sum_{i=k}^m \mbinom{i}{v} n_i \leq \mbinom{m}{v} \frac{(t-1) \binom{m-v}{s-v} - \alpha}{\binom{k-v}{s-v}}
\]
where $\alpha$ is the integer such that $\alpha \equiv (t-1)\binom{m-v}{s-v} \mod{\binom{k-v}{s-v}}$ and $0 \leq \alpha < \binom{k-v}{s-v}$.
\end{theorem}

When $s=2$, the only choice for $v$ in the above is 1 and we recover the family of constraints used in \cite{CheHorMam}. However, for larger values of $s$, there are many choices for $v$ and we obtain a wider variety of additional constraints. By computing optimal values of the linear program resulting from adding these constraints we are able to improve on the best known upper bounds on $z(m,n;s,t)$ for some small parameter sets. Our computations indicate that by far the most useful choice of $v$ in Theorem~\ref{T:constraint} is $v=s-1$. By theoretically analysing the linear programs with Roman's constraints and one of the additional $v=s-1$ constraints, we are able to obtain a closed form upper bound on $z(m,n;s,t)$.

\begin{theorem}\label{T:main}
Let $s$, $t$, $m$ and $n$ be integers with $2 \leq s \leq m$ and $2 \leq t \leq n$. For any $k$ such that $\max\{2,s^2-2s\} \leq k \leq m$, we have
$z(m,n;s,t) \leq B_k(m,n;s,t)$ where
\[B_k(m,n;s,t) = \frac{\binom{m}{s-1}}{\binom{k}{s-1}} \biggl(\mfrac{(t-1)(m-s+1)}{s}\left(\mfrac{\beta(k+1)}{k-s+1}+1\right) - \mfrac{\alpha \beta(k-s+2)}{k-s+1}\biggr) + \mfrac{(k+1)(s-1-\beta)}{s}n,\]
$\alpha$ is the integer such that $\alpha \equiv (t-1)(m-s+1) \mod{k-s+1}$ and $0 \leq \alpha < k-s+1$, and
\[\beta = \mfrac{(s-1)(k-s+1)-\alpha(s-1)}{(k+1)(k-s+1)-\alpha(s-1)}.\]
\end{theorem}

Note that the denominator of the expression given for $\beta$ is positive since $s \leq k$ and $\alpha < k-s+1$. Despite the complicated appearance of the bound in Theorem~\ref{T:main}, for any fixed $s$, $t$ and $m$, each choice of $k$ simply gives a linear function of $n$. Throughout the paper we often visualise bounds by treating them as functions of $n$ for fixed $s$, $t$ and $m$.

The classical upper bound on Zarankiewicz numbers is given by the K\"{o}vari-S\'{o}s-Tur\'{a}n theorem \cite{Hyl,KovSosTur} which states that $z(m,n;s,t)<(t-1)^{1/s}m n^{1-1/s}+(s-1)n$. More recently, Nikiforov \cite{Nik} proved that, for each $k \in \{0,\ldots,s-2\}$, we have  $z(m,n;s,t)<(t-k-1)^{1/s}m n^{1-1/s}+(s-1)n^{1+k/s}+km$. These bounds improve and generalise both the K\"{o}vari-S\'{o}s-Tur\'{a}n theorem and an important upper bound of F\"{u}redi \cite{Fur}. However, in many cases Roman's bounds, given in Theorem~\ref{T:Rom}, are still the best known upper bounds on Zarankiewicz numbers. When $m$ and $n$ are large in comparison with $s$ and $t$, the best of Roman's bounds is asymptotic to $(t-1)^{1/s}m n^{1-1/s}$ (obtained when $k$ is asymptotic to $(t-1)^{1/s}m n^{-1/s}$). This outperforms the K\"{o}vari-S\'{o}s-Tur\'{a}n theorem, but may be inferior to the best of Nikiforov's bounds (by a constant factor) when $m \geq n^{2/s}$. Roman's bounds are generally superior to Nikiforov's in cases where $s$, $t$, $m$ and $n$ are small.
The only improvements to Roman's and Nikiforov's bounds that we are aware of come from the already-discussed bounds of \cite{CheHorMam} in some cases with $s=2$, from results of \cite{DamHegSzo} in some ``design-like'' cases with $s=t=2$, and from work concentrating on small parameter sets \cite{ColRiaWalRad,GodHenOel,Tan}. A large amount of work has concentrated on showing that these upper bounds are tight, either asymptotically or up to constant factors: see  \cite{AloRonSza,AloMelMubVer,Con,Mor} for example. Finally, there are a few results showing that these upper bounds are exactly achieved in various cases \cite{CheHorMam,Cul,DamHegSzo,Guy} as well as the work on small parameter sets just mentioned.

We arrange the remainder of the paper as follows. In Section~\ref{S:HypergraphRoman} we discuss how Roman's bounds operate. Theorem~\ref{T:constraint} generalises the extra constraints used in \cite{CheHorMam} so they apply to cases where $s>2$; we prove Theorem~\ref{T:constraint} in Section~\ref{S:constraints}. We discuss the computational results for small parameter sets that we obtain using these constraints in Section~\ref{S:computational}. In Section~\ref{S:main} we prove the closed form upper bound given in Theorem~\ref{T:main} by adding a single extra $v=s-1$ constraint of this form to the linear program.  We give some final thoughts and open questions in Section~\ref{S:conclusion}.

\section{Roman's bound \label{S:HypergraphRoman}}


For a given $s$, $t$ and $m$, the function of $n$ given by the least of Roman's bounds is piecewise linear and has a critical point $(t-1)\binom{m}{s}/\binom{\ell}{s}$ with critical value $\ell(t-1)\binom{m}{s}/\binom{\ell}{s}$ for each integer $\ell \geq s$. We refer to these critical points as \emph{Roman points} throughout the rest of the paper. At any (integral) Roman point, it is known that $z(m,n;s,t)$ achieves the bound if and only if there exists an \emph{$s$-$(m,k,t-1)$-design}: a $k$-uniform hypergraph on $m$ vertices in which any set of $s$ vertices is a subset of exactly $t-1$ edges. Furthermore, it is known \cite{Cul} that \eqref{E:RomanBound} with $k=s-1$ holds with equality for all $n \geq (t-1)\binom{m}{s}$. That is, whenever $n$ is at least the value of the largest Roman point, the best of Roman's bounds is in fact achieved by $z(m,n;s,t)$. See Figure~\ref{F:Roman} for a visualisation of a segment of the piecewise linear function given by the best of Roman's bounds. For an interesting generalisation of the result of \cite{Cul} to a multipartite variant of the Zarankiewicz problem, see \cite{MulNag}.

\begin{figure}[h!]
    \centering
    \includegraphics[width=0.6\linewidth]{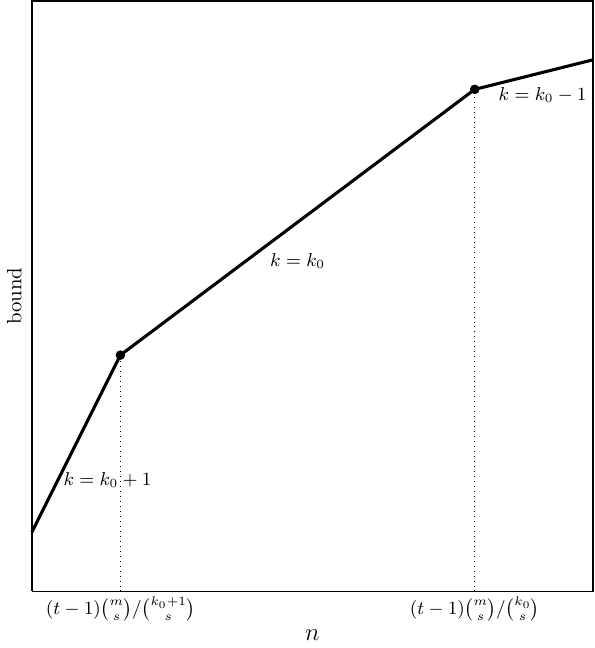}
    \caption{A portion of the piecewise linear function given by the best of Roman’s
bounds: the equation of each linear segment is given by the right hand side of \eqref{E:RomanBound} with $k \in \{k_0-1,k_0,k_0+1\}$ chosen as indicated where $k_0 \geq s$. Note that the change in gradients has been exaggerated for clarity.}
    \label{F:Roman}
\end{figure}

Recall that in \cite{CheHorMam} it was shown that, in the case $s = 2$, adding one family of extra constraints to Roman's linear program can give a better upper bound. The improvements occur for values of $n$ close to some Roman
points.

\section{Proof of Theorem~\ref{T:constraint} \label{S:constraints}}

We first introduce a few more definitions related to hypergraphs. Let $H$ be a hypergraph and consider a subset of its vertices, $X \subseteq V(H)$. For a given edge $E \in E(H)$, we say that $X$ and $E$ are \emph{incident} with each other if $X \subseteq E$.  The \emph{degree} of $X$, denoted $\deg_H(X)$, is the number of edges of $H$ incident with $X$. The \emph{total $s$-degree} of $H$ is the sum of $\deg_H(S)$ over all $s$-subsets $S$ of $V(H)$. The total $1$-degree is simply the total degree.

The special case of Theorem~\ref{T:constraint} where $s=2$ and $v=1$ was proved in \cite[Lemma 2.4]{CheHorMam}, using a notion of the deficiency of a hyperedge. We prove Theorem~\ref{T:constraint} along similar lines, but we require a generalised definition of deficiency and some more complication in the proof.

\begin{definition}
For an edge $E$ of a hypergraph and positive integers  $k,s,v$ such that $v < s$, we define the \emph{$(k,s,v)$-deficiency} of $E$ to be 0 if $|E| \geq k$ and otherwise to be
\begin{equation} \nonumber
\mbinom{k-v}{s-v} - \mbinom{|E| - v}{s-v}.
\end{equation}
\end{definition}

For an edge $E$ of size less than $k$ and a fixed $v$-subset $U$ of $E$, the $(k,s,v)$-deficiency of $E$ counts the number of incidences with $s$-supersets of $U$ that would be gained if $E$ were replaced by a superset of $E$ with size $k$.
If the values of $k$, $s$, and $v$ are clear from context, we may simply write \emph{deficiency}.

We are now ready to prove Theorem~\ref{T:constraint}. In fact, we prove a lemma that states that the numbers of edges of each size in a hypergraph $H$ with certain properties must obey linear constraints given by \eqref{E:LP1}, \eqref{E:LP2} and \eqref{E:LPgen}. The constraints \eqref{E:LP1} and \eqref{E:LP2} are those of Roman's linear program, which we already discussed, and we state and prove them here for convenience. The new family of constraints given by \eqref{E:LPgen} are exactly those claimed by Theorem~\ref{T:constraint}. Throughout the paper we let $\alpha$ be as defined in Theorem~\ref{T:constraint}, where the values of $m$, $s$, $t$, $v$ and $k$ will always be clear from context.

\begin{lemma}\label{L:obeyLP}
Let $H$ be an $(s,t-1)$-linear hypergraph $H$ of order $m$ with $n$ edges, each of which has size at least $s-1$. If $H$ has exactly $n_i$ edges of size $i$ for each $i \in \{s-1,\ldots,m\}$, then $(n_{s-1},\ldots,n_m)$ satisfies
\begin{align}
\sum_{i=s-1}^m n_i &= n \label{E:LP1} \\
\sum_{i=s-1}^m \mbinom{i}{s} n_i &\leq (t-1) \mbinom{m}{s} \label{E:LP2} \\
\intertext{and, for all positive integers $v$ and $k$ such that $v<s \leq k \leq m$,}
\frac{1}{ \binom{k-v}{s-v} - \alpha } \sum_{i=s-1}^{k-1} \left( \mbinom{i-v}{s-v} - \alpha \right) \mbinom{i}{v} n_i + \sum_{i=k}^m \mbinom{i}{v} n_i &\leq \mbinom{m}{v} \frac{(t-1) \binom{m-v}{s-v} - \alpha}{\binom{k-v}{s-v}}.\label{E:LPgen}
\end{align}
\end{lemma}

\begin{proof}
We have that \eqref{E:LP1} holds because $H$ has $n$ edges. An edge of size $i$ contributes $\binom{i}{s}$ to the total $s$-degree of $H$ and by definition the total $s$-degree of an $(s,t-1)$-linear hypergraph cannot exceed $(t-1)\binom{m}{s}$. Thus \eqref{E:LP2} holds.

Fix $v \in \{1,\ldots,s-1\}$, $k \in \{s,\ldots,m\}$ and a set $X$ of $v$ vertices. Let $n_i(X)$ denote the number of edges of size $i$ incident with $X$. Let $c$ be the integer such that $c\binom{k-v}{s-v} = (t-1) \binom{m-v}{s-v} - \alpha$. Recall that $(t-1) \binom{m-v}{s-v} - \alpha \equiv 0 \mod{\binom{k-v}{s-v}}$, hence $c$ is indeed an integer.

We will next show that
\begin{align}
\deg_H(X) &\leq c + \frac{\tau(X)}{\binom{k-v}{s-v} - \alpha}, \label{bound deg alpha bottom}
\end{align}
where we define $\tau(X)$ to be the sum of the $(k,s,v)$-deficiencies of all edges incident to $X$. For this purpose we can assume $\deg_H(X) \geq c+1$ for otherwise \eqref{bound deg alpha bottom} holds trivially.

The number of $s$-subsets of $V$ that contain $X$ is $\binom{m-v}{s-v}$ and each of these is incident with at most $t-1$ edges of $H$. On the other hand, an edge of size $i$ containing $X$ is incident with exactly $\binom{i-v}{s-v}$ of these $s$-subsets for each $i \in \{s-1,\ldots,m\}$ (in particular, an edge of size $s-1$ is incident with none of them). Thus, by the definition of $c$,
\begin{align*}
 c \mbinom{k-v}{s-v} + \alpha = (t-1) \mbinom{m-v}{s-v} & \geq \sum_{i=s-1}^m \mbinom{i-v}{s-v} n_i(X) \\
&\geq \sum_{i=k}^m \mbinom{k-v}{s-v} n_i(X) + \sum_{i=s-1}^{k-1} \mbinom{i-v}{s-v} n_i(X) \\
&= \mbinom{k-v}{s-v} \sum_{i=s-1}^m n_i(X) + \sum_{i=s-1}^{k-1} \left(\mbinom{i-v}{s-v} - \mbinom{k-v}{s-v}\right)n_i(X) \\
&= \mbinom{k-v}{s-v} \deg_H(X) - \tau(X).
\end{align*}
Thus, using $\binom{k-v}{s-v} \geq 1$  and then our assumption that $\deg_H(X) \geq c + 1$,
\[
\tau(X) \geq \mbinom{k-v}{s-v} \left(\deg_H(X) - c \right) - \alpha
\geq \left(\mbinom{k-v}{s-v} - \alpha\right)\left(\deg_H(X) - c\right)\]
and rearranging gives \eqref{bound deg alpha bottom} as desired.

Now, using \eqref{bound deg alpha bottom},
\begin{equation} \label{E:totalvDeg}
\sum_{i=s-1}^m \mbinom{i}{v} n_i = \sum_{X \in {\scriptscriptstyle\binom{V}{v}}} \deg_H(X) \leq \sum_{X \in {\scriptscriptstyle\binom{V}{v}}} \left( c + \frac{\tau(X)}{\binom{k-v}{s-v} - \alpha} \right) = \mbinom{m}{v}c + \frac{1}{ \binom{k-v}{s-v} - \alpha } \sum_{X \in {\scriptscriptstyle\binom{V}{v}}}\tau(X),
\end{equation}
where $V$ denotes the vertex set $V(H)$. Observe that, by the definition of $\tau$,
\begin{equation} \label{E:totalDeficiency}
\sum_{X \in {\scriptscriptstyle\binom{V}{v}}}\tau(X) = \sum_{i=s-1}^{k-1} \left(\mbinom{ k-v }{ s-v } - \mbinom{ i - v }{ s - v }\right) \mbinom{i}{v} n_i.
\end{equation}
Substituting \eqref{E:totalDeficiency} into \eqref{E:totalvDeg} and rearranging we obtain
\[
\sum_{i=s-1}^m \mbinom{i}{v} n_i - \frac{1}{ \binom{k-v}{s-v} - \alpha} \sum_{i=s-1}^{k-1} \left( \mbinom{ k-v }{ s-v } - \mbinom{ i-v }{ s-v } \right) \mbinom{i}{v} n_i \leq \mbinom{m}{v} c.\]
Substituting for $c$ and consolidating the terms corresponding to $i \in \{s-1,\ldots,k-1\}$ from both sums yields \eqref{E:LPgen}.
\end{proof}

\section{Computational results \label{S:computational}}

Suppose there exists an $(s,t-1)$-linear hypergraph with $m$ vertices and $n$ edges which has exactly $n_i$ edges of size $i$ for each $i \in \{s-1,\ldots,m\}$, and no edges of size less than $s-1$. In Lemma~\ref{L:obeyLP} we have established that the inequalities \eqref{E:LP1}, \eqref{E:LP2} and \eqref{E:LPgen} hold for all $v \in \{1,\ldots,s-1\}$ and $k \in \{s,\ldots,m\}$.

For positive integers $s$, $t$, $m$ and $n$, let $\mathcal{E}(m,n;s,t)$ be the linear program on nonnegative real variables $n_{s-1}, \dots, n_m$ that aims to maximise $\sum_{i=s-1}^m in_i$ subject to constraints \eqref{E:LP1}, \eqref{E:LP2} and, for all $v \in \{1,\ldots,s-1\}$ and $k \in \{s,\ldots,m\}$, \eqref{E:LPgen}. By our discussion above in Sections \ref{S:HypergraphRoman} and \ref{S:constraints}, the optimal value of $\mathcal{E}(m,n;s,t)$ is an upper bound on $z(m,n;s,t)$. For small parameters $(m,n;s,t)$ it is possible to compute optimal values for $\mathcal{E}(m,n;s,t)$. This section outlines the results we obtained from doing so. The linear programs $\mathcal{E}(m,n;s,t)$ from Lemma~\ref{L:obeyLP} were implemented in the Python programming language using the linear program solver GLPK available through SageMath \cite{Sage}.

We observed that very often the optimal value of $\mathcal{E}(m,n;s,t)$ is matched by the optimal value of the simpler linear program, denoted $\mathcal{E}^*(m,n;s,t)$, whose constraints are \eqref{E:LP1}, \eqref{E:LP2} and, for $v=s-1$ and all $k \in \{s,\ldots,m\}$, \eqref{E:LPgen}. For example, Table~\ref{table:s-1 vs all} shows that for $s \in \{3,4,5\}$, $t \in \{s, \ldots, 5\}$, $m \in \{s, \ldots, 60\}$ and $n \in \{m, \ldots, 60\}$, the simpler linear program $\mathcal{E}^*(m,n;s,t)$ and has the same optimal value as the program $\mathcal{E}(m,n;s,t)$ in about 97.5\% of the cases (in contrast, using a different single choice for $v$ achieves the same optimal value as $\mathcal{E}(m,n;s,t)$ in only approximately $30\%$ of the cases).  This motivates the choice of linear program (restricted to $v=s-1$) that we use to establish Theorem~\ref{T:main} in Section~\ref{S:main} below. Theorem~\ref{T:main} is itself less powerful than $\mathcal{E}^*(m,n;s,t)$ because it only uses \eqref{E:LPgen} with $v=s-1$ and the single `best' choice of $k$, whereas $\mathcal{E}^*(m,n;s,t)$ uses \eqref{E:LPgen} with $v=s-1$ and all $k \in \{s,\ldots,m\}$ simultaneously.

\begin{table}[h!]
\begin{center}
\begin{tabular}{c|c|c|c|c}
\centering
$s$ & $t$ & $\#$ cases  & $\#$ where $\mathcal{E}^*(m,n;s,t)$ matches
& $\#$ where Theorem~\ref{T:main} matches
\\ \hline
3 & 3 & 1711 & 1697 & 1334 \\
3& 4 & 1711	& 1696 & 1354\\
3 & 5 & 1711 & 1693 & 1455 \\
4& 4 & 1653 & 1597 & 618 \\
4 & 5 & 1653 & 1565 & 797 \\
5 & 5 & 1596 & 1538 & 209
\end{tabular}
\caption{Comparing Theorem~\ref{T:main}, $\mathcal{E}^*(m,n;s,t)$ and $\mathcal{E}(m,n;s,t)$ for $m \in \{s, \ldots, 60\}$ and $n \in \{m, \ldots, 60\}$.  The third column gives the total number of cases, the fourth column gives the number of cases where the simpler linear program (with $v=s-1$) $\mathcal{E}^*(m,n;s,t)$ has the same optimal value as $\mathcal{E}(m,n;s,t)$, and the last column gives the number of cases in which that same optimal value is given by the closed form bound in Theorem~\ref{T:main}.}
\label{table:s-1 vs all}
\end{center}
\end{table}

Tables~\ref{table-3-3}-\ref{table-5-5} in Appendix~\ref{A:tables} show where Theorem~\ref{T:main} or optimal values from the linear program $\mathcal{E}(m,n;s,t)$ improve on the best previously known bound for $(s,t) = (3,3)$, $(3,4)$, $(3,5)$, $(4,4)$, $(4,5)$, and $(5,5)$ respectively. Here, the best known upper bound is often the Roman bound, but for small values of $n$ and $m$, when $s=t$, Tan \cite{Tan} catalogued the best known values, indicating which are known to be exact. For all the parameter sets covered by these tables we confirmed that the Roman bound is at least as good as the bound of Nikiforov \cite{Nik}.

We referred to \cite{Tan} for $(s,t) = (3,3)$ and $(s,t) = (4,4)$ when developing Table~\ref{table-3-3} and Table~\ref{table-4-4}, respectively, but everywhere we made an improvement, the quoted value in the table of \cite{Tan} was the Roman bound.  We use $x^{*}$ to indicate that the full linear program is needed to obtain the bound $x$, either because the bound of Theorem~\ref{T:main} cannot be used or because Theorem~\ref{T:main} does not give the best result. Most of the improved bounds are an improvement by 1, except for the bold values, which improve the previous (Roman) bound by 2.

The closed form nature of the bound of Theorem~\ref{T:main} allows us to use it for much larger (and much more unbalanced) values of $m$ and $n$. To give some idea of its behaviour in these regimes, in Appendix~\ref{A:plots} we give a number of plots that show the improvement (if any) that Theorem~\ref{T:main} provides over the Roman bound given by Theorem~\ref{T:Rom}. In each plot the values of $s$, $t$ and $m$ are fixed, the value of $n$ is plotted on the horizontal axis and the improvement on the vertical axis. We consider $(s,t) \in \{(3,3),(3,4),(3,5),(4,4),(4,5),(5,5)\}$ and $m \in \{21,22,23,24\}$ where the latter was chosen somewhat arbitrarily as a range to showcase the different behaviour for different values of $m$. For visual clarity, in calculating the improvement, we treat the bounds as real functions of a real variable $n$ and do not take the floor of either. For each plot, there are no improvements for values of $n$ larger than those pictured. As was the case for the $s=2$ bound in \cite{CheHorMam}, the largest improvements tend to be for values of $n$ close to a Roman point. We indicate the Roman points with light vertical lines. We do not claim that the full range of behaviour of the bound of Theorem~\ref{T:main} is captured by the example choices of $s$, $t$ and $m$ for which we give plots. However, these examples do serve to demonstrate the variety of the behaviour. Again, the Roman bounds are always at least as good as those of Nikiforov for the values covered by these plots.

\section{Proof of Theorem~\ref{T:main} \label{S:main}}

In this section we prove Theorem~\ref{T:main}. The calculations become messy and we regularly employed symbolic computational software to perform tedious simplifications. Throughout this section we let $\alpha$ and $\beta$ be as defined in Theorem 1.2 (this is consistent with our previous definition of $\alpha$ since we are implicitly taking $v=s-1$ here). For convenience, we restate Theorem~\ref{T:main} below.\\

\noindent \textup{\textbf{Theorem~\ref{T:main}.}} \emph{
Let $s$, $t$, $m$ and $n$ be integers with $2 \leq s \leq m$ and $2 \leq t \leq n$. For any $k$ such that $\max\{2,s^2-2s\} \leq k \leq m$, we have
$z(m,n;s,t) \leq B_k(m,n;s,t)$ where
\[B_k(m,n;s,t) = \frac{\binom{m}{s-1}}{\binom{k}{s-1}} \biggl(\mfrac{(t-1)(m-s+1)}{s}\left(\mfrac{\beta(k+1)}{k-s+1}+1\right) - \mfrac{\alpha \beta(k-s+2)}{k-s+1}\biggr) + \mfrac{(k+1)(s-1-\beta)}{s}n,\]
$\alpha$ is the integer such that $\alpha \equiv (t-1)(m-s+1) \mod{k-s+1}$ and $0 \leq \alpha < k-s+1$, and
\[\beta = \mfrac{(s-1)(k-s+1)-\alpha(s-1)}{(k+1)(k-s+1)-\alpha(s-1)}.\]
}

\begin{proof}
First, note that by specialising \eqref{E:LPgen} to the case $v=s-1$ we obtain
\begin{equation}\label{E:LP3}
 \sum_{i=s-1}^{k-1} \mfrac{i-s+1 - \alpha}{ k-s+1 - \alpha } \mbinom{i}{s-1} n_i + \sum_{i=k}^m \mbinom{i}{s-1} n_i \leq \mbinom{m}{s-1} \mfrac{(t-1)(m-s+1) - \alpha}{k-s+1}.
\end{equation}

For each integer $k \in \{s,\ldots,m\}$, let $\mathcal{E}^*_k(m,n;s,t)$ be the linear program on nonnegative real variables $n_{s-1},\ldots,n_m$ that aims to maximise $\sum_{i=s-1}^m in_i$ subject to \eqref{E:LP1}, \eqref{E:LP2} and \eqref{E:LP3}.

Let
\begin{align}
A &= \mfrac{(k+1)(s-1-\beta)}{s} \label{E:LP1factor} \\[1mm]
B &= \bigl(1-(s-1)\beta\bigr)\mbinom{k}{s-1}^{-1} \label{E:LP2factor} \\[1mm]
C &= (s-1)\beta\mbinom{k}{s-2}^{-1} = (k-s+2)\beta\mbinom{k}{s-1}^{-1}. \label{E:LP3factor}
\end{align}\medskip

We note that $\beta$ is a decreasing function of $\alpha$ and hence, since $0 \leq \alpha \leq k-s$,
\[\mfrac{s-1}{(k-s)(k-s+2)+k+1} \leq \beta \leq \mfrac{s-1}{k+1}.\]
So it is easy to see that $A$ and $C$ are nonnegative. Using the upper bound on $\beta$ we see that we are guaranteed that $B$ will also be nonnegative when $k \geq s^2-2s$.

We are interested in the inequality given by adding $A$ times \eqref{E:LP1}, $B$ times \eqref{E:LP2} and $C$ times \eqref{E:LP3}. Simplification reveals that this inequality is exactly
\begin{equation}\label{E:fgIneq}
\sum_{i=s-1}^{k-1}f(i)n_i+\sum_{i=k}^{m}g(i)n_i \leq B_k(m,n;s,t)
\end{equation}
where $f:[s-1,k-1] \rightarrow \R$ and $g:[k,m]\rightarrow \R$ are functions given by
\begin{align*}
f(i) &=\mfrac{(k+1)(s-1-\beta)}{s} + \frac{\binom{i}{s-1}}{\binom{k}{s-1}}\left(\mfrac{(1-(s-1)\beta)(i-s+1)}{s} + \mfrac{\beta (k-s+2)(i-s+1-\alpha)}{(k-s+1-\alpha)}\right) \\[2mm]
g(i) &=\mfrac{(k+1)(s-1-\beta)}{s}+\frac{\binom{i}{s-1}}{\binom{k}{s-1}}\left(\mfrac{(1-(s-1)\beta)(i-s+1)}{s}+\beta(k-s+2)\right).
\end{align*}

Thus, to show that $B_k(m,n;s,t)$ is an upper bound on the optimal value of $\mathcal{E}^*_k(m,n;s,t)$, it suffices to show that the objective function of $\mathcal{E}^*_k(m,n;s,t)$ is bounded above by the left hand side of \eqref{E:fgIneq}. This will be achieved if we can show that $f(i) \geq i$ for each $i \in \{s-1,\ldots,k-1\}$ and $g(i) \geq i$ for each $i \in \{k,\ldots,m\}$. We accomplish this in Claims~\ref{Cl:f} and \ref{Cl:g} below. There we will see that, in particular, $f(k-1)=k-1$ and $g(i)=i$ for $i \in \{k,k+1\}$. In fact, $A$, $B$ and $C$ were chosen to be the unique coefficients for which this property would hold.
\end{proof}

\begin{claim}\label{Cl:f}
Let $f$ be as defined in the proof of Theorem~$\ref{T:main}$. Then $f(k-1)=k-1$ and $f(i) \geq i$ for each $i \in \{s-1,\ldots,k-2\}$.
\end{claim}

\begin{proof}
Note that substituting $i = k-1$ into $f(i)$ and rearranging we get
\begin{align*}
f(k-1) &= k-1 + \mfrac{s-1}{k} + \beta\left( \mfrac{(k-s+1)(k-s)}{ks}\left(\mfrac{s(k-s+2)(k-s-\alpha)}{(k-s)(k-s+1-\alpha)}-s+1 \right) - \mfrac{k+1}{s}\right).
\end{align*}
Using the definition of $\beta$, one can check that $\beta$ satisfies the equation
\[\mfrac{s-1}{k} + \beta\left( \mfrac{(k-s+1)(k-s)}{ks}\left(\mfrac{s(k-s+2)(k-s-\alpha)}{(k-s)(k-s+1-\alpha)}-s+1 \right) - \mfrac{k+1}{s}\right) = 0\]
which implies $f(k - 1) = k - 1$.

We now move on to show that $f(i) \geq i$ for each $i \in \{s-1,\ldots,k-2\}$. Since $f(k-1)=k-1$, it will suffice to show that $f'(i) \leq 1$ for $s-1 \leq i \leq k-1$, where $i$ is viewed as a real variable. For integers $a$ and $b$ with $b \leq a$ we use $(a)_b$ to denote the falling factorial $a(a-1)\cdots(a-b+1)$. Noting that the derivative of $(i)_{s-1}$ with respect to $i$ is $\sum_{j=0}^{s-2}\frac{1}{i-j}(i)_{s-1}$  we have,

\begin{align}
f'(i)
&= \mfrac{(i)_{s-1}}{(k)_{s-1}} \biggl(\left(\mfrac{1-(s-1)\beta}{s}+\mfrac{\beta(k-s+2)}{k-s+1-\alpha}\right) \ + \notag \\[2mm]
& \qquad \left(\mfrac{(1-(s-1)\beta)(i-s+1)}{s} + \mfrac{\beta (k-s+2)(i-s+1-\alpha)}{(k-s+1-\alpha)}\right)\sum_{j=0}^{s-2}\mfrac{1}{i-j}\biggr) \notag \\[2mm]
&\leq \mfrac{(i)_{s-1}}{(k)_{s-1}}  \biggl(\mfrac{1-(s-1)\beta}{s}+\mfrac{\beta(k-s+2)}{k-s+1-\alpha} \ + \notag \\[2mm]
& \qquad \left(\mfrac{(1-(s-1)\beta)(i-s+1)}{s} + \mfrac{\beta (k-s+2)(i-s+1-\alpha)}{(k-s+1-\alpha)}\right)\mfrac{s-1}{i-s+2}\biggr) \notag \\[2mm]
&= \mfrac{(i)_{s-1}}{(k)_{s-1}}  \biggl(\mfrac{k(k+1)\bigl(s(i-s+1)+1\bigr)-\alpha(s-1)\bigl((k-i)s^2+s(2i-k-1)+2\bigr)} {s(i-s+2)\bigl((k+1)(k-s+1)-\alpha(s-1)\bigr)}\biggr) \label{E:fDerivBound}
\end{align}
where the inequality was obtained by using $\sum_{j=0}^{s-2}\frac{1}{i-j} \leq \sum_{j=0}^{s-2}\frac{1}{i-s+2} = \frac{s-1}{i-s+2}$ and the final equality was obtained by substituting for $\beta$ and a computer aided simplification.
The derivative of the right hand side of \eqref{E:fDerivBound} with respect to $\alpha$ is
\[-\mfrac{(k+1)(k-s+2)(s-1)^2((k-i)s-1)}{s(i-s+2)((k+1)(k-s+1)-\alpha(s-1))^2}\cdot\mfrac{(i)_{s-1}}{(k)_{s-1}}\]
which is negative and hence the expression is decreasing in $\alpha$. So we may substitute $\alpha=0$ into \eqref{E:fDerivBound} and rearrange to obtain
\[
f'(i) \leq \mfrac{ks(i-s+2)-k(s-1)} {s(i-s+2)(k-s+1)}\cdot\mfrac{(i)_{s-1}}{(k)_{s-1}}
< \mfrac{k}{k-s+1}\cdot\mfrac{(i)_{s-1}}{(k)_{s-1}} = \mfrac{(i)_{s-1}}{(k-1)_{s-1}} \leq 1.\qedhere\]
\end{proof}

\begin{claim}\label{Cl:g}
Let $g$ be as defined in the proof of Theorem~$\ref{T:main}$. Then $g(k)=k$, $g(k+1)=k+1$ and $g(i) \geq i$ for each $i \in \{k+2,\ldots,m\}$.
\end{claim}

\begin{proof}
Note that $(s-1)\beta \leq 1$ since $\beta \leq \frac{s-1}{k+1}$ and  $k \geq s^2-2s$. Thus it can be seen from the definition of $g$ that it is a constant plus the product of $\binom{i}{s-1}$ and a nondecreasing positive linear function of $i$. So $g$ is convex, since $\binom{i}{s-1}$ is a nondecreasing positive convex function of $i$ for $i \geq s-2$ and a product of two nondecreasing positive convex functions is convex. Since a convex function has at most two zeroes, to prove the claim it suffices to show that $g(k)=k$ and $g(k+1)=k+1$. This can be established by simplifying $g(i)$ as follows and making the appropriate substitutions for $i$.
\[g(i) = \mfrac{(k+1)(s-1-\beta)}{s} + \mfrac{(i)_{s-1}}{s(k)_{s-1}}\left( (1-(s-1)\beta)i + \beta(sk+1)-s+1\right).\qedhere\]
\end{proof}


\section{Conclusion\label{S:conclusion}}

Here we generalised the methods used in \cite{CheHorMam} to find new upper bounds on Zarankiewicz numbers $z(n,m;s,t)$ for general $s$. In \cite{CheHorMam} it was further shown that, for $s=2$, either the new upper bound or Roman's bound was in fact achievable in the vast majority of cases in which $n=\Theta(tm^2)$. It would be interesting to see if analogous results could be proved for $s>2$, but we suspect that attempting to employ similar arguments to those used in \cite{CheHorMam} would become intractably complicated in the general case. One could also attempt to determine whether some of our improved lower bounds for small parameter sets are achievable. Computational methods like those of \cite{Tan} could be employed here, or perhaps constructions based on finite geometric objects, like those in \cite{DamHegSzo}.

It is also interesting to note in Table~\ref{table:s-1 vs all} that Theorem~\ref{T:main} which uses only a single constraint of the form \eqref{E:LP3} is often outperformed by $\mathcal{E}^*(m,n;s,t)$ which uses all the \eqref{E:LP3} constraints simultaneously. One could attempt to find other closed form bounds that use more than one of these constraints together, but this may well come at the expense of an even more unwieldy expression.

\bigskip
\noindent\textbf{Acknowledgments.}
The second author was supported by an Australian Government Research Training Program (RTP) Scholarship. The third author was supported by Australian Research Council grants DP220102212 and
DP240101048.

\pagebreak

\appendix

\section{Tables of improved upper bounds}\label{A:tables}

Tables~\ref{table-3-3}-\ref{table-5-5} in this appendix show where Theorem~\ref{T:main} or optimal values from the linear program $\mathcal{E}(m,n;s,t)$ improve on the best previously known bound for $(s,t) = (3,3)$, $(3,4)$, $(3,5)$, $(4,4)$, $(4,5)$, and $(5,5)$ respectively. In every case where we make an improvement, this previous bound was the Roman bound. Bold indicates an improvement of 2 on the Roman bound. Bold and underlined indicates an improvement of 3. An asterisk indicates that $\mathcal{E}(m,n;s,t)$ was needed to obtain the improved bound; all other improvements are implied by Theorem~\ref{T:main}.

\begin{table}[H]
\begin{center}
\scalebox{0.9}{
\begin{tabular}{r|rrrrrrr}
\diagbox[width=12mm,height=9mm]{$m$}{$n$} &   17  &  18  &  19  &  20  &  21  &  22  &  23 \\
\hline
10 & & & & & & 111  & 115  \\
11 & & & 108  & 112  & 116  & &  \\
12 & & & & & & &   \\
13  & 116 & 121 & 125 & 130 & 135 & &  \\
14  & 124  & 129 & 135 & 140 & 145 & 150 & \\
15   & {\bf 132} & 138 & 143 & 149 & 154  & & 165   \\
16   & 141  & {\bf 146}  & {\bf 152}  & {\bf 158}  & 164  & 169  & 175   \\
\end{tabular}}
\caption{Improved upper bounds on $z(m,n;3,3)$.} 
\label{table-3-3}
\end{center}
\end{table}

\begin{table}[H]
\begin{center}
\scalebox{0.9}{
\begin{tabular}{r|rrrrrrrrrrrrrrrrr}
\diagbox[width=12mm,height=9mm]{$m$}{$n$} &  7  &  8  &  9  &  10  &  11  &  12  &  13  &  14  &  15  &  16  &  17  &  18  &  19  &  20  &  21  &  22  &  23 \\
\hline
5 & 27 & 30 & 33 & & & & & & & & & & & & & & \\
6 &  &   &  & & & & & & & & & & & & & &  \\
7  &  &   &  & & & & & & & & & & & & & &  \\
8 &    &  47 & & & & & & & & & & & & & & & \\
9 &  &   &   &   &   &  & 76 & 80 & 84 & 88 & & & & & & & \\
10 &  &   &   &   69 & & & & & & & & 107$^* $ & & & & & \\
11 &  &   &   &   &  & & &  97 & 102 & & 111 & 116 & & & 130 & & 139 \\
12 &  &   &  & & & & & & & & & & & & & &  \\
13 &  &   &   &   &   &   &   &    113 & 119 & 125 & 130 & & & & & 157 & 162 \\
14 &  &   &   &   &   &   &   &    122 & 128 & 134 & 140 & 146 & 152 & & 163$^* $ & & \\
15 &  &   &   &   &   &   &   &   &   &   143 & 149 & & & & & & 185 \\
16  &   &   &   &   &   &   &   &   &   &  152 & 159 & & & 179$^* $ & 185$^* $ & & \\
\end{tabular}}
\caption{Improved upper bounds on $z(m,n;3,4)$.}
\label{table-3-4}
\end{center}
\end{table}
\vspace{5mm}

\begin{table}[H]
\begin{center}
\scalebox{0.9}{
\begin{tabular}{r|rrrrrrrrrrrrrrrr}
\diagbox[width=12mm,height=9mm]{$m$}{$n$} &   8  &  9  &  10  &  11  &  12  &  13  &  14  &  15  &  16  &  17  &  18  &  19  &  20  &  21  &  22  &  23 \\
\hline
6 & 39 & & & & & & & & & & & & 79$^* $ & & & \\
7   & & & & & 61 & 65 & 69 & 72 & & & & & & & & \\
8    &  & & & & & & & & & & & & & & & \\
9    &   &  63 & 68 & & & & & & & & & & & & & \\
10   &   &   &  75 & & & 91 & 96 & 101 & & & & & & & & \\
11    &   &   &   &  & & & & & & & & & 136 & 141 & & \\
12    &   &   &   &   &  & & 114 & 120 & 126 & & & & & & &\\
13   &   &   &   &   &   &  & & & & & & & & 165 & & 176 \\
14   &   &   &   &   &   &   &  & & 146 & & 159 & 165$^* $ & & & & \\
15  &   &   &   &   &   &   &   &  149 & 156 & & & & & 189 & {\bf 195} & 202 \\
16   &   &   &   &   &   &   &   &   &  166 & & & & & & & \\
\end{tabular}}
\caption{Improved upper bounds on $z(m,n;3,5)$.}
\label{table-3-5}
\end{center}
\end{table}
\vspace{5mm}

\begin{table}[H]
\begin{center}
\scalebox{0.9}{
\begin{tabular}{r|rrrrrrrrrrrrrrrrrrrr}
\diagbox[width=12mm,height=9mm]{$m$}{$n$} &      14  &  15  &  16  &  17  &  18  &  19  &  20  &  21  &  22  &  23 \\
\hline
10 &   & 107 & 113$^* $  & 119$^* $  &  {\bf 124}$^* $ & 130$^* $  & 135$^* $  & 140$^* $  & 146$^* $  & 151$^* $  \\
11 &  {\bf 110} & 117 & & & & & & & & \\
12 &   120 & & & & 147 & 153 & {\bf 159} & {\bf 166} & {\bf 172} & {\bf 178} \\
13 &  & 137 & & 152$^* $  & & & & & & \\
14 &    {\bf 138} & {\bf 146} & 154 & 162 & & 177 & 184 & {\bf 191} & {\bf 198} & {\bf 205} \\
15 &    &  156 & 165 & & & & 197 & 205 & 212 & 220 \\
16 &       &   &  {\bf 175}$^*$ & 184 & 193 & 201 & 209 & 217 & {\bf 225} & {\bf 233} \\
\end{tabular}}
\caption{Improved upper bounds on $z(m,n;4,4)$.}
\label{table-4-4}
\end{center}
\end{table}
\vspace{5mm}

\begin{table}[H]
\begin{center}
\scalebox{0.9}{
\begin{tabular}{r|rrrrrrrrrrrrrrr}
\diagbox[width=12mm,height=9mm]{$m$}{$n$} &    9  &  10  &  11  &  12  &  13  &  14  &  15  &  16  &  17  &  18  &  19  &  20  &  21  &  22  &  23 \\
\hline
7  & 53$^* $ & 58$^* $ & & 67$^* $ & & & & & & & & & & & \\
8   & & & & & & & & 97$^* $ & 102$^* $ & 107$^* $ & 112$^* $ & 117$^* $ & 121$^* $ & 126$^* $ & \\
9   &  & & & & & & & & & & & & & & \\
10  &   &  81 & 88 & {\bf 94} & 101 & 107 & 113 & & & & & & & & 161$^* $ \\
11   &   &   &  96 & 103 & 110 & 117 & & & 137 & {\bf 143} & 150 & 156 & 162 & 168 & \\
12    &   &   &   &  & 120 & & 135$^* $ & 142$^* $ & 149$^* $ & & & & & & \\
13    &   &   &   &   &  {\bf 129} & {\bf 137} & 145 & 153 & & 168 & 175 & {\bf 182} & 190 & {\bf 197} & {\bf 204} \\
14    &   &   &   &   &   &  & 156 & & & & {\bf 188} & 196 & 204 & 211 & 219 \\
15   &   &   &  167 & 176 & & 193 & 201 & & & 226 & 234 \\
16  &   &   &   &   &   &   &   &  187 & & 205 & 214 & {\bf 222} & {\bf 231} & {\bf \underline{239}} & {\bf 248} \\
\end{tabular}}
\caption{Improved upper bounds on $z(m,n;4,5)$.}
\label{table-4-5}
\end{center}
\end{table}
\vspace{5mm}

\begin{table}[H]
\begin{center}
\scalebox{0.9}{
\begin{tabular}{r|rrrrrrrrrrrrrr}
\diagbox[width=12mm,height=9mm]{$m$}{$n$} &   10  &  11  &  12  &  13  &  14  &  15  &  16  &  17  &  18  &  19  &  20  &  21  &  22  &  23 \\
\hline
8 & 69$^* $ & 75$^* $ & 81$^* $ & & 92$^* $ & & & & & & & & & \\
9 & & & & & & & & & & 135$^* $ & 141$^* $ & 147$^* $ & {\bf 153}$^* $ & {\bf 159}$^* $\\
10  &  & & & & & & & & 143$^* $ & 149$^* $ & & & & \\
11   &   &  & 109$^* $ & 117$^* $ & {\bf 124}$^* $ & {\bf 132}$^* $ & 140$^* $ & 147$^* $ & 155$^* $ & 162$^* $ & & & & \\
12  &   &   &  119$^* $ & {\bf 127}$^* $ & {\bf135}$^* $ & 144$^* $ & 152$^* $ & 160$^* $ & 168$^* $ & 176$^* $ & & 192$^* $ & 200$^* $ & {\bf 207}$^* $ \\
13   &   &   &   &  & & & & 173$^* $ & & & 199 & & & \\
14  &   &   &   &   &  157$^* $ & {\bf 166}$^* $ & 176$^* $ & {\bf 185}$^* $ & {\bf 194}$^* $ & 204$^* $ & 213$^* $ & & & \\
15    &   &   &   &   &   &  {\bf 178}$^* $ & {\bf 188}$^* $ & {\bf 197}$^* $ & {\bf 207}$^* $ & 217$^* $ & 227$^* $ & 237$^* $ & & 256$^* $ \\
16   &   &   &   &   &   &   &  200$^* $ & 211$^* $ & 221$^* $ & & 242$^* $ & 252$^* $ & 263$^* $ & \\
\end{tabular}}
\caption{Improved upper bounds on $z(m,n;5,5)$.}
\label{table-5-5}
\end{center}
\end{table}

\section{Plots of improvements on Roman's bound}\label{A:plots}

In this appendix we give plots that show the improvement (if any) that Theorem~\ref{T:main} provides over the bound of Roman given by Theorem~\ref{T:Rom}. In each plot the values of $s$, $t$ and $m$ are fixed, the value of $n$ is plotted on the horizontal axis and the improvement on the vertical axis. We consider $(s,t) \in \{(3,3),(3,4),(3,5),(4,4),(4,5),(5,5)\}$ and $m \in \{21,22,23,24\}$. We treat the bounds as real functions of a real variable $n$ and do not take the floor of either. For each plot, there are no improvements for values of $n$ larger than those pictured. We indicate the Roman points with light vertical lines.

\begin{center}
\begin{tabular}{ccc}
  $m$ & $(s,t)=(3,3)$ & $(s,t)=(3,4)$ \\[5mm]
  \raisebox{2.5cm}{21} & \includegraphics[width=7.5cm]{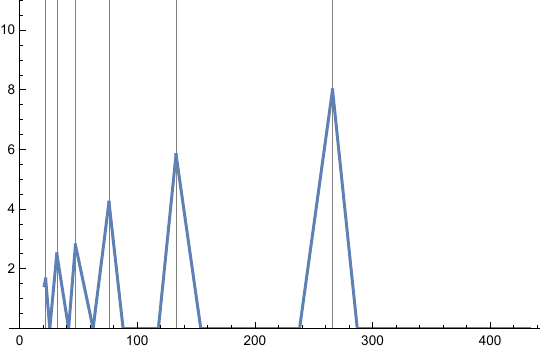} & \includegraphics[width=7.5cm]{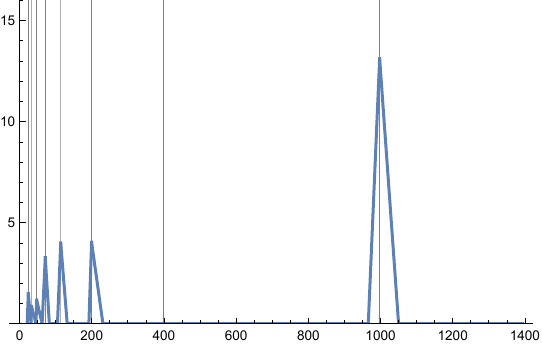} \\[5mm]
  \raisebox{2.5cm}{22} & \includegraphics[width=7.5cm]{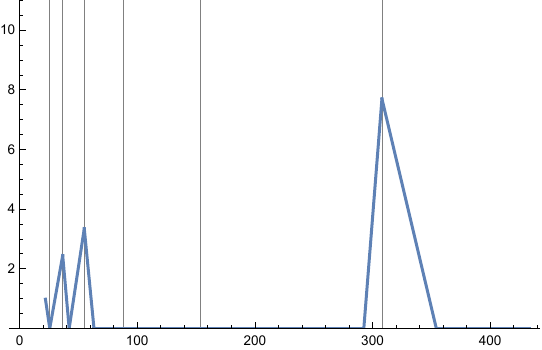} & \includegraphics[width=7.5cm]{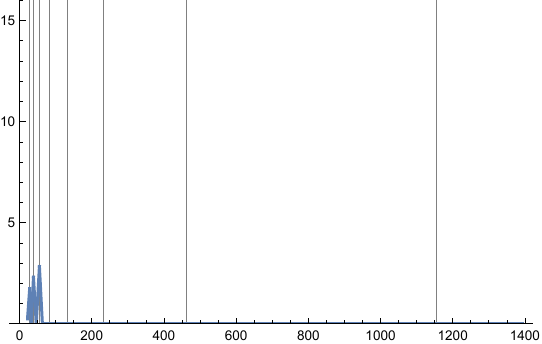} \\[5mm]
  \raisebox{2.5cm}{23} & \includegraphics[width=7.5cm]{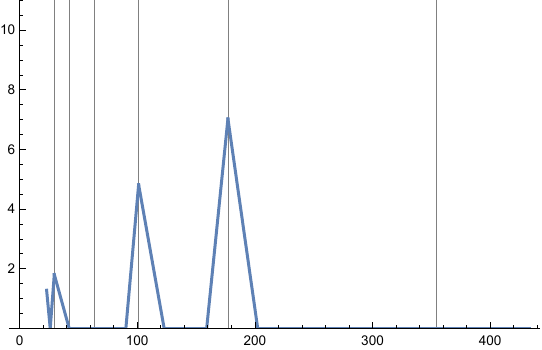} & \includegraphics[width=7.5cm]{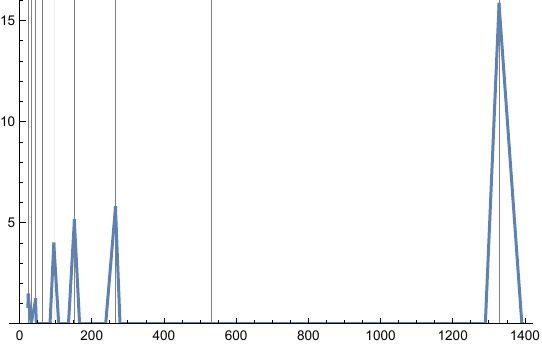} \\[5mm]
  \raisebox{2.5cm}{24} & \includegraphics[width=7.5cm]{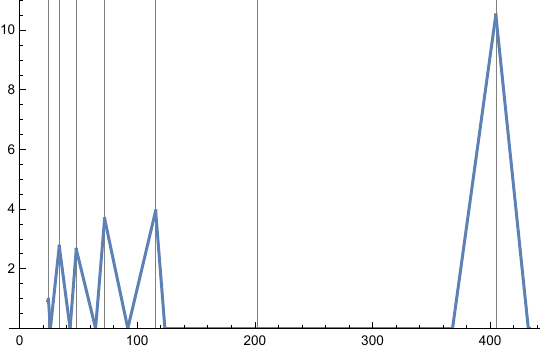} & \includegraphics[width=7.5cm]{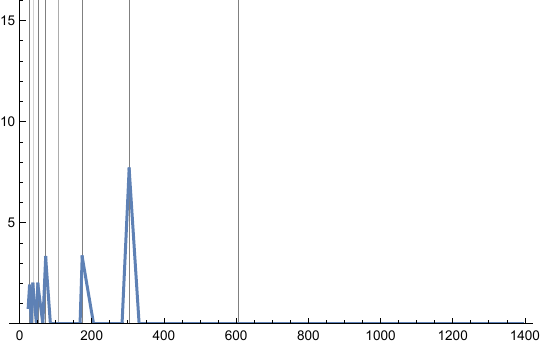}
\end{tabular}
\end{center}

\begin{center}
\begin{tabular}{ccc}
  $m$ & $(s,t)=(3,5)$ & $(s,t)=(4,4)$ \\[5mm]
  \raisebox{2.5cm}{21} & \includegraphics[width=7.5cm]{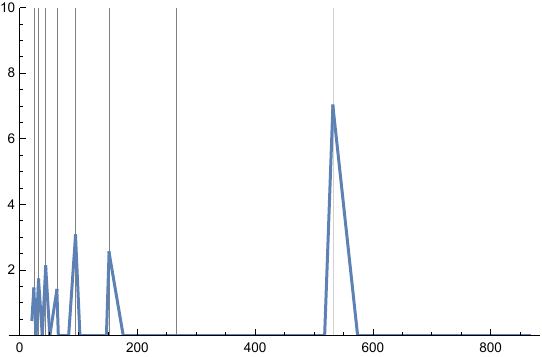} & \includegraphics[width=7.5cm]{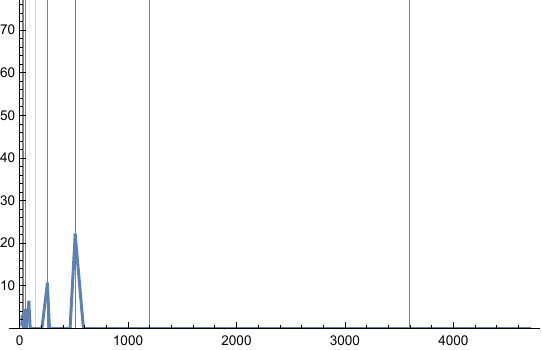} \\[5mm]
  \raisebox{2.5cm}{22} & \includegraphics[width=7.5cm]{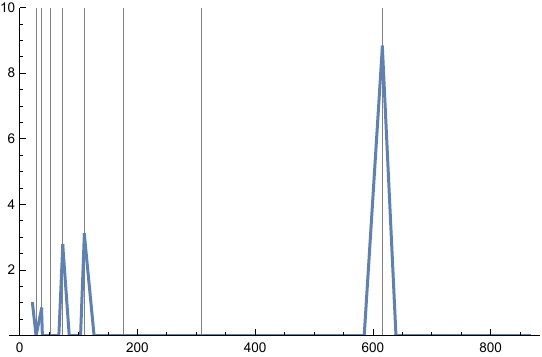} & \includegraphics[width=7.5cm]{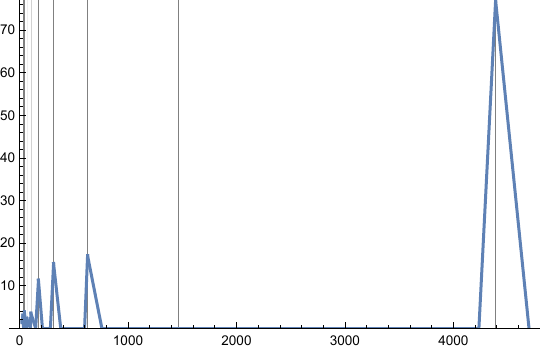} \\[5mm]
  \raisebox{2.5cm}{23} & \includegraphics[width=7.5cm]{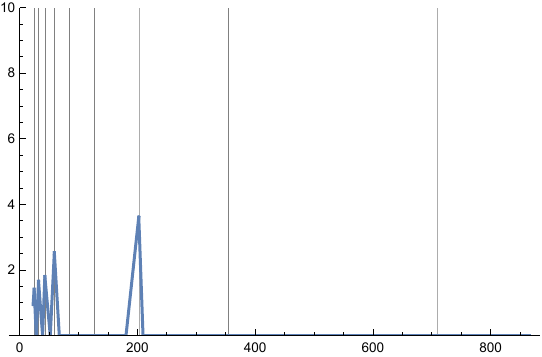} & \includegraphics[width=7.5cm]{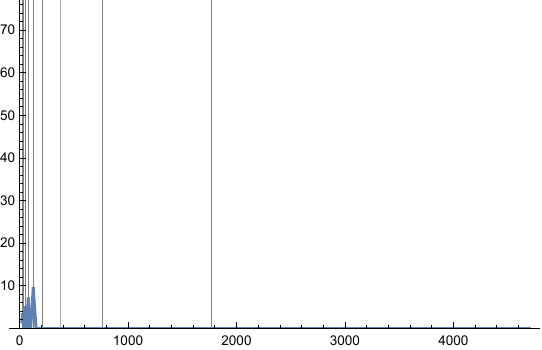} \\[5mm]
  \raisebox{2.5cm}{24} & \includegraphics[width=7.5cm]{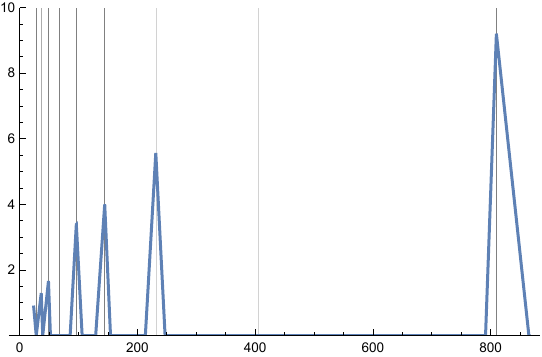} & \includegraphics[width=7.5cm]{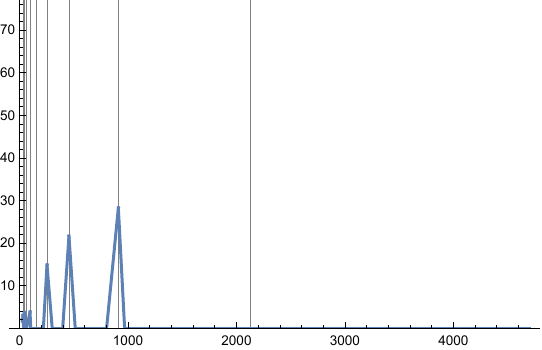}
\end{tabular}
\end{center}

\begin{center}
\begin{tabular}{ccc}
  $m$ & $(s,t)=(4,5)$ & $(s,t)=(5,5)$ \\[5mm]
  \raisebox{2.5cm}{21} & \includegraphics[width=7.5cm]{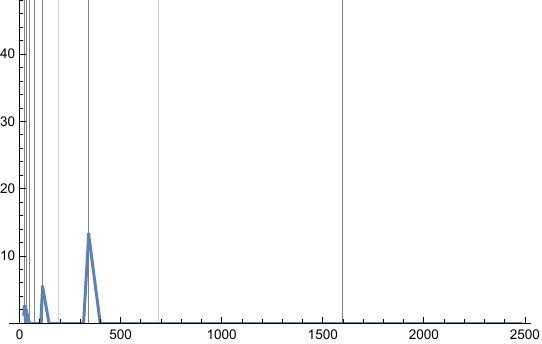} & \includegraphics[width=7.5cm]{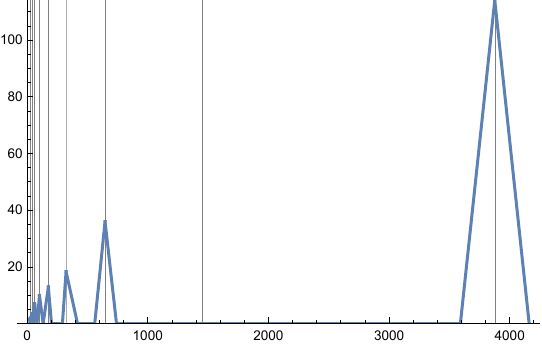} \\[5mm]
  \raisebox{2.5cm}{22} & \includegraphics[width=7.5cm]{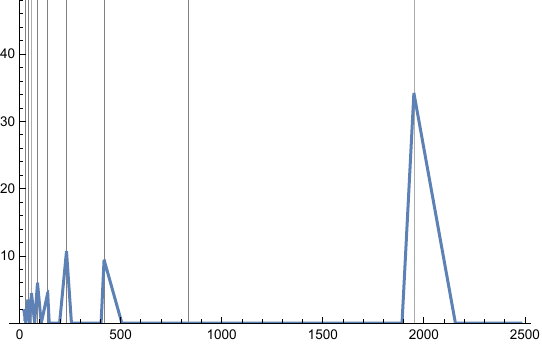} & \includegraphics[width=7.5cm]{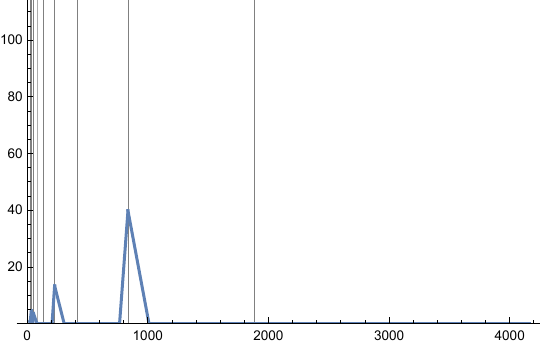} \\[5mm]
  \raisebox{2.5cm}{23} & \includegraphics[width=7.5cm]{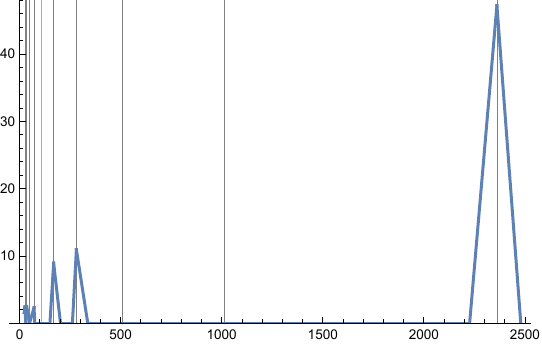} & \includegraphics[width=7.5cm]{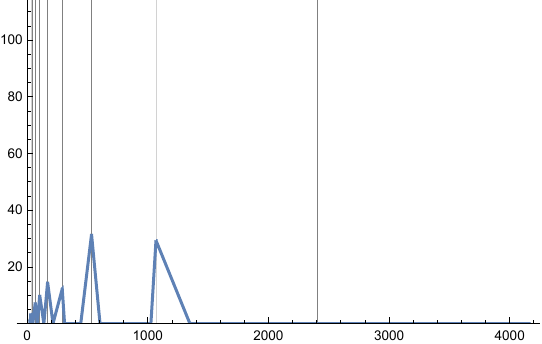} \\[5mm]
  \raisebox{2.5cm}{24} & \includegraphics[width=7.5cm]{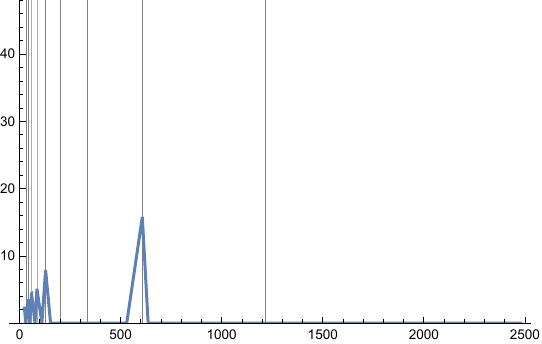} & \includegraphics[width=7.5cm]{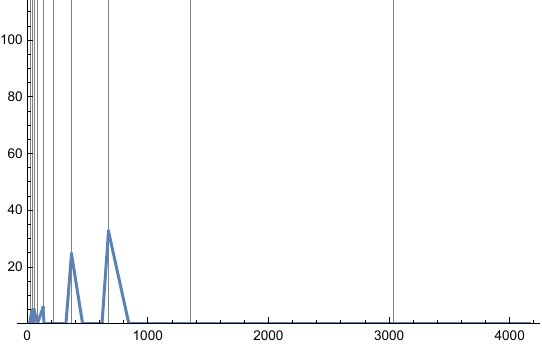}
\end{tabular}
\end{center}

\end{document}